\documentclass[11pt]{article}
\usepackage{geometry}
\usepackage{amssymb, amsmath, amsthm, fullpage, color, hyperref}

\newtheorem{theorem}{Theorem}[section]
\newtheorem{definition}[theorem]{Definition}
\newtheorem{proposition}[theorem]{Proposition}
\newtheorem{lemma}[theorem]{Lemma}
\newtheorem{corollary}[theorem]{Corollary}
\newtheorem{example}[theorem]{Example}

\theoremstyle{remark}
\newtheorem*{remark}{Remark}

\numberwithin{equation}{section}

\newcommand{\Ppp}{\mathbb P}

\newcommand{\NN}{\mathbb N}

\makeatletter
\newcommand*{\defeq}{\mathrel{\rlap{%
                     \raisebox{0.3ex}{$\m@th\cdot$}}%
                     \raisebox{-0.3ex}{$\m@th\cdot$}}%
                     =}
\makeatother

\newcommand{\hz}{\widehat{0}}

\DeclareMathOperator{\vdW}{vdW}
\DeclareMathOperator{\gcdtr}{gcdtr}
\DeclareMathOperator{\lcm}{lcm}

\newcommand{\onethingatopanother}[2]{\genfrac{}{}{0pt}{}{#1}{#2}}

\begin{document}

\title{The van der Waerden complex}
\author{Richard Ehrenborg, Likith Govindaiah, \\
Peter S.\ Park and Margaret Readdy}
\date{} 

\maketitle

\begin{abstract}
We introduce the van der Waerden complex $\vdW(n,k)$ 
defined as the simplicial complex whose facets correspond 
to arithmetic progressions
of length $k$ in the vertex set $\{1, 2, \ldots, n \}$. We show 
the van der Waerden complex $\vdW(n,k)$
is homotopy equivalent to
a $CW$-complex whose cells asymptotically
have dimension at most $\log k / \log \log k$.
Furthermore, we give bounds on $n$ and $k$ which imply
that the van der Waerden complex is contractible. 
\end{abstract}

\section{Introduction}

A number of recent papers have considered the topology of
cell complexes
associated to number theoretic concepts.
These include Bj\"orner's study of the Betti
numbers and homotopy type of a simplicial complex 
whose Euler
characteristic is the negative of the Mertens function,
as well as a related $CW$-complex whose
Euler characteristic is the summatory Liouville function~\cite{Bjorner},
Musiker and Reiner's work describing the coefficients of the cyclotomic
polynomial as the torsion homology 
of a sequence of simplicial complexes~\cite{Musiker_Reiner},
and Pakianathan and Winfree's topological reformulations of
number theoretic conjectures
using threshold complexes~\cite{Pakianathan_Winfree}.

Recall that
van der Waerden's 
Theorem from Ramsey theory states that given 
positive integers~$k$ and~$r$, there is
an integer $M = M(k,r)$ so that
when the integers $1$ through $n$ with $n \geq M$
are colored
with $r$ colors, there is a monochromatic
arithmetic progression of length $k$~\cite{van_der_Waerden}.
There is an upper bound for $M(k,r)$ due to Gowers
coming out of his proof of Szemeredi's theorem~\cite{Gowers}.

Motivated by van der Waerden's theorem,
we
define the {\em van der Waerden complex} $\vdW(n,k)$
to be the simplicial complex on the vertex
set $\{1, 2, \ldots, n\}$ whose
facets correspond to all arithmetic progressions of length $k$,
that is, the facets have the form
$\{ x, x+d, x + 2 \cdot d, \ldots, x + k \cdot d\}$,
where $d$ is a positive integer and
$1 \leq x < x + k \cdot d \leq n$.

Observe that the van der Waerden
complex $\vdW(n,k)$
is a simplicial complex of dimension $k$,
that is, each facet has dimension $k$.
Furthermore, when $k=1$,
the complex $\vdW(n,1)$
is the complete graph $K_{n}$,
which is homotopy equivalent to a wedge
of $\binom{n-1}{2}$ circles.

This paper is concerned with understanding the topology
of the van der Waerden complex.
By constructing a discrete Morse matching,
we show that the dimension of the homotopy type
of the van der Waerden
complex is bounded
above by the maximum of the number of distinct primes factors
of all positive integers less than or equal to $k$.
See Theorem~\ref{theorem_main}. 
This bound is asymptotically described as $\log k / \log \log k$.
See Theorem~\ref{theorem_asymptotic}.
In Section~\ref{section_contractible} we
give bounds under which the van der Waerden complex
is contractible. We then look at the implications of our results 
when studying the topology of the family of van der Waerden complexes 
$\vdW(5k,k)$ where $k$ is any positive integer.
We end with open questions in the concluding remarks.

\section{Preliminaries}
\label{section_preliminaries}

Let $\Ppp$ be the set of positive integers.
Let $[n]$ denote the set $\{1,2, \ldots, n\}$
and $[i,j]$ denote the interval $\{i,i+1, \ldots, j\}$.

\begin{definition}
The {\em van der Waerden complex} $\vdW(n,k)$
is the simplicial complex on the vertex
set~$[n]$ whose
facets correspond to all arithmetic progressions of length $k$ in $[n]$,
that is, the facets have the form
$$
     \{ x, x+d, x + 2 \cdot d, \ldots, x + k \cdot d\},  
$$
where $d$ is a positive integer and
$1 \leq x < x + k \cdot d \leq n$.
\end{definition}
We remark that the van der Waerden complex is not
monotone in the variable $k$.
For instance, the set $\{1,4,7\}$ is a facet in
$\vdW(7,2)$, but
is not a face of $\vdW(7,3)$.

Let $P$ be a finite partially ordered set (poset) with 
partial order relation $\prec$. For further information on
posets, see~\cite[Chapter~3]{Stanley_EC_I}.
A \textit{matching} $M$ on $P$ is a collection
of disjoint pairs of elements of $P$
such that $x \prec y$ for each pair
$(x,y) \in M$.
Given a matching $M$, we define the two partial functions
$u$ and $d$ on $P$ as $u(x) = y$ and
$d(y) = x$  when $(x,y) \in M$.
\begin{definition}
A matching $M$ on a poset $P$ is {\em a discrete Morse matching}
if there do not exist elements $x_{1}, x_{2}, \ldots, x_{j}$
such that
$$  x_{1} \prec u(x_{1}) \succ x_{2} \prec u(x_{2}) \succ
\cdots
     \succ x_{j} \prec u(x_{j}) \succ x_{1}  .  $$ 
\end{definition}
One can describe this definition in more concrete terms. 
In the Hasse diagram of the poset~$P$,
direct all of the edges upward except for
those in the matching~$M$.
These are directed downward.   The matching~$M$ is
a discrete Morse matching if and only if the resulting directed graph has
no directed cycles.
Hence a discrete Morse matching is also known as an {\em acyclic} matching.

One way to construct an acyclic matching is by taking the union of 
acyclic matchings on subposets as detailed below. 
For details, see~\cite[Theorem~11.10]{Kozlov_book}.

\begin{theorem}[Patchwork Theorem]
Let $\varphi : P \longrightarrow Q$
be an order-preserving map between the two posets~$P$ and $Q$.
Suppose that a discrete Morse matching exists on each
subposet $\varphi^{-1}(q)$ for every $q \in Q$.
Then the union of these matchings is itself
a discrete Morse matching on $P$.
\label{theorem_patchwork}
\end{theorem}

An element in the poset $P$ that does not belong
to matching $M$ is called {\em critical}. We
can now state the main theorem of 
discrete Morse theory~\cite[Theorem 2.5]{Forman_2}.

\begin{theorem}[Forman]
\label{theorem_Forman}
Let $\Delta$ be a simplicial complex and let $M$
be a discrete Morse matching on the face poset of
$\Delta - \{\emptyset\}$.
Let $c_{i}$ denote the number of critical elements
of dimension $i$.
Then the simplical complex $\Delta$ is homotopy equivalent
to a $CW$-complex consisting of $c_{i}$ cells of dimension
$i$ for $i \geq 0$.
\end{theorem}

A $CW$-complex consisting
of one $0$-dimensional
cell and $c$ $i$-dimensional
cell is homotopy equivalent
to a wedge of $c$
$i$-dimensional spheres.
Combining this observation
with Theorem~\ref{theorem_Forman},
we have the following result. 
\begin{proposition}
Let $\Delta$ be a simplicial complex and let $M$
be a discrete Morse matching on $\Delta - \{\emptyset\}$.
If there only are 
$c$ critical cells of dimension
$i>0$ and one critical
cell of dimension $0$ then $\Delta$ is homotopy equivalent
to a wedge of $c$
spheres each of dimension $i$.
\label{proposition_wedge_of_spheres}
\end{proposition}

For a detailed treatment of the theory of $CW$-complexes,
we direct the reader to~\cite{Hatcher}.

\section{A bound on the dimension of the homotopy type}
\label{section_upper_bound}

We begin by studying a family of sets
that we use as building blocks in order to
understand the van der Waerden complex.
\begin{definition}
Define $\Gamma(k)$ to be the family of subsets given by
$$  \Gamma(k)
   =
      \{  F \: : \: \{0,k\} \subseteq F \subseteq [0,k], \: \gcd(F) = 1 \}  ,  $$
where $[0,k] = \{0, 1, \ldots, k\}$ and $\gcd(F)$ denotes the greatest common divisor of the elements of~$F$. 
\end{definition}
In other words, $\Gamma(k)$ is the family
of all subsets containing the pair $\{0,k\}$ and contained in the interval $[0,k]$
such that the only arithmetic sequence from $0$ to $k$ containing this family
is the entire interval $[0,k]$.
Even though $\Gamma(k)$ is not a simplicial
complex, we refer to the sets in the collection~$\Gamma(k)$ as faces.

\begin{lemma}
Let $k$ be a positive integer which is not squarefree.
Then the family $\Gamma(k)$ has a discrete Morse matching
with no critical cells.
\label{lemma_non_squarefree}
\end{lemma}
\begin{proof}
Let $s$ denote the product of all of the distinct primes dividing $k$.
Since $k$ is not squarefree, we have that $s < k$.

Let $F$ be a face in $\Gamma(k)$ such that
$s \not\in F$.
Note that $\gcd(F \cup \{s\}) = \gcd(1,s) = 1$.
Thus the union $F \cup \{s\}$ also belongs to the family $\Gamma(k)$.
Similarly, if we have a face $F$ in $\Gamma(k)$ such that $s \in F$,
we claim that $F - \{s\}$ is also a face of $\Gamma(k)$.
Assume that the prime $p$ divides $\gcd(F - \{s\})$.
It follows from $k \in F -\{s\}$ that $p | k$.
Since~$s$ is the squarefree part of $k$,
the prime~$p$ also divides $s$.
This yields the contradiction that
$p$ divides $\gcd(\gcd(F - \{s\}), s) = \gcd(F) = 1$.
Hence we conclude that
the set difference $F - \{s\}$ is 
also in the family $\Gamma(k)$, proving the claim.

Given a face $F$, match it with
the symmetric difference $F \triangle \{s\}$.
This defines an acyclic matching
on $\Gamma(k)$ with no critical cells,
as claimed.
\end{proof}

\begin{proposition}
Let $k$ be a squarefree positive integer.
There exists a discrete Morse matching on the family $\Gamma(k)$
which has only one critical cell given by
$$ 
\{0,k\} \cup \{k/q \:\: : \:\: q|k \text{ and $q$ prime}\}.  
$$
\label{proposition_squarefree}
\end{proposition}
\begin{proof}
The proof is by induction on the number of prime factors of $k$.
If there are no prime factors, that is, $k = 1$, then
$\Gamma(1)$ only consists of the set $\{0,1\}$
which is then the only critical cell.

Assume now that $k$ has at least one prime factor
$p$.
We begin to create the matching as follows.
Let~$F$ be a set such that
$k/p \not\in F$ and $\gcd(F) = 1$.
Note that $\gcd(F) = 1$ implies that
$\gcd(F \cup \{k/p\}) = \gcd(\gcd(F), k/p) = 1$.
Hence the two sets $F$ and $F \cup \{k/p\}$
are faces of $\Gamma(k)$.
Match these two faces.
Note that among the elements that have been matched
so far, there is no directed cycle.

The remaining unmatched faces 
are of the form $F \cup \{k/p\}$
where $k/p \not\in F$, $\gcd(F) \neq 1$
and $\gcd(F \cup \{k/p\}) = 1$.
Observe that
$\gcd(\gcd(F), k/p) = 1$. Since $\gcd(F)$ divides~$k$,
the only option for $\gcd(F)$ is the prime $p$.
The unmatched faces are described by
$F \cup \{k/p\}$ where
 $k/p \not\in F$ and $\gcd(F) = p$.
Note these sets $F$ are in bijection with
faces in the family $\Gamma(k/p)$ by sending
$F$ to the set $F/p = \{i/p \: : \: i \in F\}$,
that is,
the faces of $\Gamma(k)$
which have not yet been matched
are the image of the map
$\varphi(G) = p \cdot G \cup \{k/p\}$
applied to the family~$\Gamma(k/p)$.

By the induction hypothesis we can find
a Morse matching for $\Gamma(k/p)$ where
the only critical cell is given by
$$  
     C 
     =
     \{0,k/p\} \cup \{k/pq \:\: : \:\: q|{k}/{p} \text{ and $q$ prime}\}    .
$$
By applying the map $\varphi$
to the Morse matching of $\Gamma(k/p)$ we obtain a matching
on the unmatched cells of $\Gamma(k)$.
The only remaining unmatched cell is
$$  
     \varphi(C) 
     =
     \{0,k\} \cup \{k/q \:\: : \:\: q|k/p \text{ and $q$ prime}\} 
       \cup \{k/p\} .
$$
This is the claimed unmatched face.

By induction there are no directed
cycles in the image $\varphi(\Gamma(k/p))$.
Furthermore, since the image $\varphi(\Gamma(k/p))$
lies in ranks which are less than
that of the elements which were first matched, there
can be no directed cycles between these two parts.
Hence we conclude the matching constructed is a Morse matching.
This completes the induction hypothesis and the proof.
\end{proof}

As an application, 
the matching constructed in Lemma~\ref{lemma_non_squarefree}
and
Proposition~\ref{proposition_squarefree} yields a combinatorial identity 
for the number theoretic M\"obius function.
Recall for a positive integer~$k$
that the M\"obius function is defined by  
$$
     \mu(k) = \begin{cases}
             	(-1)^{r}
                 & \text{if $k$ is square-free and has $r$ prime factors,} \\
                 0
                 & \text{if $k$ is not square-free.}
    	      \end{cases}
$$

\begin{corollary}
For $k$ a positive integer 
the M\"obius function $\mu(k)$ is given by
$$   \mu(k) = \sum_{F \in \Gamma(k)} (-1)^{|F|} . $$
\end{corollary}
\begin{proof}
The matching of
Lemma~\ref{lemma_non_squarefree}
and
Proposition~\ref{proposition_squarefree}
yields a sign-reversing involution on
the set~$\Gamma(k)$.
The only unmatched face
$F = \{0,k\} \cup \{k/q \:\: : \:\: q|k \text{ and $q$ prime}\}$
occurs when~$k$ is squarefree, and
it satisfies
\begin{align*}
(-1)^{|F|}
& =
(-1)^{\text{\# prime factors of $k$}} = \mu(k) .
\qedhere
\end{align*}
\end{proof}

For $1 \leq x < y \leq n$, define the set 
$D(n,k,x,y)$ by
$$
D(n,k,x,y)
  =
\{ d \in \Ppp \:\: : \:\: d|y-x \text{ and } 
\{x, x+d, \ldots, y\} \in \vdW(n,k) \} .
$$
The following inclusion holds.
\begin{equation}
D(n,k,x,y)
\subseteq
\{ d \in \Ppp \:\: : \:\: d|y-x \text{ and } 
 (y-x)/d \leq k \} . 
\label{equation_D}
\end{equation}
In general this inclusion is not an equality 
since the right-hand side does not depend on $n$.
The following example illustrates this phenomenon. 

\begin{example}
{\rm
Consider the case  $n=7$ and $k=3$. We have that $\{1,7\}$ is an edge
in $\vdW(7,3)$
since it is contained in the facet $\{1,3,5,7\}$.
In fact, $D(7,3,1,7) = \{2,6\}$,
whereas the right-hand side of equation~\eqref{equation_D}
is given by the set $\{2,3,6\}$.
}
\end{example}

The following notion will be useful in what follows.

\begin{definition}
For a finite non-empty set $S = \{s_{1} < s_{2} < \cdots < s_{j}\}$ 
of positive integers,
define the {\em greatest common divisor after translation} as
$$
     \gcdtr(S) =
                \gcd(s_{2} - s_{1}, s_{3} - s_{1}, \ldots, s_{j} - s_{1}).
$$
\end{definition}

\noindent
Observe that $\gcdtr(S)$ is
the largest divisor $d$ of
the difference $\max(S)-\min(S)$
for which $\{\min(S), \min(S)+d, \ldots, \max(S)-d, \max(S)\}$ 
is an arithmetic progression that contains $S$.

\begin{proposition}
The van der Waerden complex 
$\vdW(n,k)$
can be expressed as the disjoint union
$$ \vdW(n,k)
       \: \: = \: \:
     \{\emptyset, \{1\}, \{2\}, \ldots, \{n\}\}
     \: \stackrel{\cdot}{\cup} \:
     \bigcup_{\onethingatopanother{1 \leq x < y \leq n}
                                     {d \in D(n,k,x,y)}}^{\cdot}
       (d \cdot \Gamma((y-x)/d) + x)
$$
where
$d \cdot \Gamma(j) + x$
denotes the family
$\{ d \cdot F + x : F \in \Gamma(j)\}$
and
$d \cdot F + x$ denotes the set
$\{d \cdot i + x : i \in F\}$.
\label{proposition_disjoint_union}
\end{proposition}
\begin{proof}
Let $F$ be a face in $\vdW(n,k)$.
If the face has cardinality at most $1$, then it belongs to the family
$\{\emptyset, \{1\}, \{2\}, \ldots, \{n\}\}$.
Otherwise, it has cardinality at least $2$.
For such a face, let
$x = \min(F)$,
$y = \max(F)$
and
$d = \gcdtr(F)$.
Note that
the face~$F$ is a subset of $d \cdot \Gamma((y-x)/d) + x$.
This yields the decomposition.
\end{proof}

We now use the aforementioned acyclic matching 
to obtain bounds for the dimension of the cells of
a $CW$-complex which is homotopy equivalent to 
the complex $\vdW(n,k)$. 
Our bounds are stated in terms of the 
{\em primorial function} $\Pi(x)$, which is defined as
the product
\[
\Pi(x) = \prod_{p \leq x} p,
\]
where $p$ ranges over all primes less than or equal to $x$.
Let $p_{i}$ denote the $i$th prime.
The product of the $r$ first primes
is then given by $\Pi(p_{r})$.

\begin{theorem}
Let $n$, $k$ and $r$ be positive integers such that
$k < \Pi(p_{r})$. 
Then the van der Waerden complex $\vdW(n,k)$
is homotopy equivalent to a $CW$-complex 
whose cells have dimension at most $r$.
\label{theorem_main}
\end{theorem}
\begin{proof}
The case $k = 1$ follows
directly from the fact that~$\vdW(n,1)$
is the complete graph~$K_{n}$.
Consider the case when $k>1$.
Define the poset $Q$ by
$$
Q
   =
\{(x,y,d) \in \NN^3  \:\: : \:\:
     1 \leq x < y \leq n, \:
         y-x \geq 2, \: d \in D(n,k,x,y) \}
   \cup
  \{\hz\},
$$ 
with partial order 
defined by
$(x,y,d) \leq_{Q} (x^{\prime}, y^{\prime},d^{\prime})$
if $x^{\prime} \leq x < y \leq y^{\prime}$
and
$d^{\prime} | d$;
and
$\hz \leq_{Q} z$ for all $z \in Q$.
We also define the map $\varphi : \vdW(n,k) - \{\emptyset\} \longrightarrow Q$ as
\[
\varphi(F) = \begin{cases}
        	(\min(F), \max(F), \gcdtr(F))
            & \text{if $\max(F) - \min(F) \geq 2$,} \\
           \hz
            & \text{otherwise.}
    	\end{cases}
\]
Note that $\varphi$ is order-preserving.
Also, 
for $(x,y,d) \in Q$,
the fiber
$\varphi^{-1}((x,y,d))$
is given by
the family
$d \cdot \Gamma((y-x)/d) + x$.
Finally,
the fiber $\varphi^{-1}(\hz)$
consists of
singletons and the pairs
$\{i,i+1\}$, where $1 \leq i < n$.

The fibers of the map $\varphi$ yield the 
decomposition of
the complex  $\vdW(n,k)$ in Proposition~\ref{proposition_disjoint_union}.
Furthermore, the fiber $\varphi^{-1}(\hz)$ has
an acyclic matching, where
the singleton $\{i\}$ is matched
with the pair $\{i,i+1\}$.
The other fibers each have an acyclic matching by 
Lemma~\ref{lemma_non_squarefree}
and
Proposition~\ref{proposition_squarefree}.
Finally, we apply
Theorem~\ref{theorem_patchwork} to conclude that
the union of the aforementioned acyclic matchings is itself an acyclic matching
on all of~$\vdW(n,k)$. 

The singleton
$\{n\}$ is the only critical cell
in the fiber $\varphi^{-1}(\hz)$. 
If $(y-x)/d$ is not squarefree
then there are no critical cells
in the fiber
$\varphi^{-1}((x,y,d))$.
On the other hand, if $(y-x)/d$ is squarefree then
the fiber
$\varphi^{-1}((x,y,d))$ has exactly one critical cell
with cardinality at most $2$ plus the number
of distinct primes factors of
$(y-x)/d$.
Since the inequality 
$(y-x)/d \leq k < \Pi(p_{r})$ holds,
the number of prime factors 
of $(y-x)/d$ is at most $r-1$.
Hence the size of the critical cells in our acyclic matching 
is at most $r+1$,
that is, their dimensions are at most~$r$.
\end{proof}

\begin{corollary}
If $k < \Pi(p_{r})$ then
the $i$th reduced homology
group~$\widetilde{H}_{i}(\vdW(n,k))$
of the van der Waerden complex vanishes
for $i \geq r+1$.
\label{corollary_homology}
\end{corollary}

We end this section by considering the
asymptotic behavior of $r$ as $k$ tends to infinity.
Recall
\begin{equation}
 \lim_{r \longrightarrow \infty}
         \frac{\log \Pi(p_{r})}{r \cdot \log \log \Pi(p_{r})}
     =
          1  ; 
\label{equation_Hardy_Wright}
\end{equation}
see the beginning of Section~22.10
in~\cite{Hardy_Wright} just before Theorem~430.

\begin{theorem}
Let $r(k)$ be the unique positive integer $r$ 
such that the inequalities
$\Pi(p_{r-1}) \leq k < \Pi(p_{r})$ hold.
Then the following asymptotic result holds:
$$ r(k) \sim \frac{\log k}{\log\log k} $$
as $k \longrightarrow \infty$.
\end{theorem}
\begin{proof}
Note that the function $\log x/\log\log x$
is increasing for $x > e^{e}$.
Hence for sufficiently large~$k$
(take $k \geq 30$)
we have the two inequalities
$$
\frac{r-1}{r}
\cdot 
\frac{\log \Pi(p_{r-1})}{(r-1) \cdot \log \log \Pi(p_{r-1})}
\leq
\frac{\log k}{r \cdot \log \log k}
<
\frac{\log \Pi(p_{r})}{r \cdot \log \log \Pi(p_{r})}
,
$$
where $r = r(k)$.
Note that $r(k) \longrightarrow \infty$ as $k \longrightarrow \infty$.
Hence by equation~\eqref{equation_Hardy_Wright}
both the right-hand and left-hand sides of the
above bound tend to $1$
as $k \longrightarrow \infty$.
Thus the asymptotic follows by the squeeze theorem.
\end{proof}

We can now state an asymptotic version of
our homotopy equivalence dimension bound.
\begin{theorem}
\label{asymptotic}
The van der Waerden complex $\vdW(n,k)$ 
is homotopy equivalent to a $CW$-complex
whose cells have dimension
asymptotically at most
$\log k / \log\log k$.
\label{theorem_asymptotic}
\end{theorem}

\section{Bounds for contractible complexes}
\label{section_contractible}

Recall that for a prime $p$ and integer $s$
that $p^{r} || s$ means that $r$ is the largest
power of $p$ dividing~$s$, that is,
$p^{r} | s$ but $p^{r+1} \not| s$.

\begin{proposition}
For an integer $a > 1$, let $L(a) = \lcm(1,2,\ldots,a)$.
Let $p$ be a prime such that
$p^{\alpha} || L$ for $\alpha \geq 1$.
Let $k$ be an integer greater than or equal to
$L(a)/p^{\alpha-1}$,
and let $n$ be an integer bounded by $a \cdot k < n \leq (a+1) \cdot k$.
Then the van der Waerden complex  $\vdW(n,k)$
is contractible.
\label{proposition_L}
\end{proposition}
\begin{proof}
Define the poset $Q$ as
$$
Q
   =
\{(x,y) \in \NN^2 \:\: : \:\:
     1 \leq x < y \leq n, \:
         y-x \geq 2 \}
   \cup
  \{\hz\},
$$ 
with the partial order relation given by
$(x,y) \leq_{Q} (x^{\prime}, y^{\prime})$
if
$x^{\prime} \leq x < y \leq y^{\prime}$,
and
$\hz \leq_{Q} z$ for all $z \in Q$.
We also define the map
$\varphi : \vdW(n,k) - \{\emptyset\} \longrightarrow Q$ as
\[
\varphi(F) = \begin{cases}
        	(\min(F), \max(F))
            & \text{if $\max(F) - \min(F) \geq 2$,} \\
           \hz
            & \text{otherwise.}
    	\end{cases}
\]
Note that $\varphi$
is order-preserving.

As in the proof of
Theorem~\ref{theorem_main},
the fiber $\varphi^{-1}(\hz)$
consists of
singletons and the pairs
$\{i,i+1\}$ such that $1 \leq i < n$.
Again, we match the singleton
$\{i\}$ with the pair $\{i,i+1\}$,
leaving us with
the critical cell $\{n\}$.

We will now describe a matching
on the fiber $\varphi^{-1}((x,y))$ for $(x,y) \in Q$. 
We have that $\{x,y\}$ is an
edge of~$\vdW(n,k)$, and that $y-x \geq 2$. 
Note that all facets of 
$\vdW(n,k)$
have steps at most $a$,
since for any $b > a$
the difference between
$z + b \cdot k$ and
$z$ is greater than $n-1$.

Let $S$ be the set
of steps of the facets containing the edge $\{x,y\}$,
that is,
\begin{align*}
   S
    & =
  \{d \in \Ppp \: : \:
  \{x,y\} \subseteq \{z, z+d, \ldots, z + k \cdot d\} \subseteq [n]\} .
\end{align*}
Note that $S$ is a non-empty subset of $[a]$.
Next let $T$ be the set of minimal elements of $S$
with respect to divisibility.
We have that $d | y-x$ for all $d \in T$.
Thus the least common multiple
$\lcm(T)$ also divides $y-x$.
Also note that
$\lcm(T) \leq \lcm(S) \leq L(a) \leq p^{\alpha - 1} \cdot k$.

We claim that $y-x \neq \lcm(T)$.
To reach a contradiction, suppose that $y-x = \lcm(T)$.
Note that $1 \not\in T$, since if the element $1$ did belong to $T$,
the set $T$ would be the singleton $\{1\}$, implying
that $y-x = \lcm(T) = 1$, which contradicts $y-x \geq 2$.
Factor $L(a)$ as $p^{\alpha} \cdot c$ where $c$ is not divisible by $p$.
Assume now that $y-x$ factors as
$p^{\beta} \cdot c^{\prime}$
such that $p$ does not divide $c^{\prime}$.
Since $y-x = \lcm(T)$ divides~$L(a)$,
we have that $c^{\prime} | c$ and
$\beta \leq \alpha$.
We have two cases to consider.

The first case is $\beta \leq 1$.
We have the following string of inequalities:
$$  y - x = p^{\beta} \cdot c^{\prime}
             \leq p^{\beta} \cdot c
             \leq p \cdot c
           =  \frac{L(a)}{p^{\alpha-1}}  \leq k  .    $$
This shows that
the edge $\{x,y\}$ is contained
in an edge of step $1$, yielding the contradiction $1 \in T$.

The second case is $\beta \geq 2$. We then have
$$  y - x = p^{\beta} \cdot c^{\prime}
             \leq p^{\beta} \cdot c
             = p^{\beta} \cdot \frac{L(a)}{p^{\alpha}}
             = p^{\beta-1} \cdot \frac{L(a)}{p^{\alpha-1}}
             \leq  p^{\beta-1} \cdot k  .    $$

We now claim that there is a facet with step $p^{\beta-1}$
that contains the edge $\{x,y\}$.
Since~$p^{\alpha}$ divides~$L(a)$,
we have the inequality
$p^{\alpha} \leq a$.
We have the following string of inequalities.
\begin{align*}
(k+1) \cdot p^{\beta-1}
 =
    k \cdot p^{\beta-1} + p^{\beta-1} & \leq
    k \cdot p^{\beta-1} + k \cdot (p-1) \cdot p^{\beta-1} \\
& =
    k \cdot p^{\beta} \leq
    k \cdot p^{\alpha} \leq
    k \cdot a
<
    n.
\end{align*}
Hence the arithmetic sequence
$F = \{p^{\beta-1}, 2 \cdot p^{\beta-1}, \ldots, (k+1) \cdot p^{\beta-1}\}$,
which has step size $p^{\beta-1}$ and length $k$, is contained
in the vertex set $[n]$, that is, $F$ is indeed a facet.
By shifting this facet $F$, we obtain all the facets
of step $p^{\beta-1}$:
$$   F - (p^{\beta-1}-1),  \: F - (p^{\beta-1}-2), \: \ldots, \: F, \: 
    \ldots, \: F + (n-(k+1) \cdot p^{\beta-1})  .  $$
Since $p^{\beta-1}$ divides $y-x$
and $y-x \leq k \cdot p^{\beta-1}$, 
one of the above facets contains the edge $\{x,y\}$,
proving the claim.

Hence the element $p^{\beta-1}$ belongs to $S$.
Either $p^{\beta-1}$ is a minimal element
of~$S$ or it is not.
Let~$p^{\gamma}$ be an element of $T$ that is less than
or equal to $p^{\beta-1}$ in the divisor order,
that is,  $\gamma \leq \beta-1$.
Consider another element $t$ in $T$,
that is, $t \neq p^{\gamma}$.
We have that $p^{\delta} \| t$
implies that $\delta < \gamma$.
Hence $p^{\gamma}$ is the largest $p$-power that divides~$\lcm(T)$.
Since $\gamma \leq \beta - 1 < \beta$,
this contradicts the fact that $p^{\beta}$ is the smallest power of $p$
dividing $\lcm(T) = y-x$.

This proves our claim. Thus we conclude that
$\lcm(T) < y-x$.

We now continue to construct our matching on the fiber $\varphi^{-1}((x,y))$ as follows.
For a face $F \in \vdW(n,k)$ such that 
$\varphi(F) = (x,y)$, we
match $F$ with the symmetric difference $F \triangle \{x+\lcm(T)\}$.
This is well-defined because if the sum $x+\lcm(T)$ does not
belong to the face $F$, then $F \cup \{x+\lcm(T)\}$
is also a face of the complex. Note that this matching 
on the fiber $\varphi^{-1}((x,y))$
is acyclic and has no critical cells.

Finally, we apply
Theorem~\ref{theorem_patchwork} to conclude that the union of the aforementioned acyclic matchings is an acyclic matching
on $\vdW(n,k)$
with the unique critical cell
$\{n\}$. It follows that
the complex
$\vdW(n,k)$ is contractible.
\end{proof}

Directly we have the following result.
\begin{proposition}
Let $a >1$ be an integer, and define $L(a)$ as
in Proposition~\ref{proposition_L}. Let
$p_{1}^{\alpha_{1}} \cdot p_{2}^{\alpha_{2}} \cdots p_{m}^{\alpha_{m}}$
be the prime factorization of $L(a)$.
Let $M(a)$ be the maximum of the set
$$\left\{p_{1}^{\alpha_{1}-1}, 
p_{2}^{\alpha_{2}-1}, \ldots, 
p_{m}^{\alpha_{m}-1}\right\} . $$
Let $k$ be an integer greater than or equal to $L(a)/M(a)$,
and let $n$ be an integer bounded by $a \cdot k < n \leq (a+1) \cdot k$.
Then the van der Waerden complex $\vdW(n,k)$
is contractible.
\label{proposition_M}
\end{proposition}

\begin{remark}
{\rm
The prime $p$ in  $M(a) = p^{\alpha-1}$ must be $2$ or~$3$.
To see this, suppose that $p \geq 5$.
Since $2^{2} < p$, we have that
$p^{\alpha-1} < 2^{\beta-1} < 2^{\beta} < p^{\alpha}$, 
so there are at least two $2$-powers
between $p^{\alpha-1}$ and~$p^{\alpha}$.
Note that $2^{\beta} < p^{\alpha} \leq a$. Hence
the maximum $M(a)$ cannot be
given by $p^{\alpha-1}$.
}
\end{remark}

\begin{lemma}
View $L$ and $M$ in the previous proposition
as functions in the variable $a$. Then
the quotient $L/M$ is a weakly increasing function in $a$.
\label{lemma_L_M}
\end{lemma}
\begin{proof}
First consider when $a$ is not a prime power.
We have $L(a-1) = L(a)$ and hence
$M(a-1) = M(a)$.
Thus the two quotients $L(a-1)/M(a-1)$ and $L(a)/M(a)$
are equal.

Next assume that
$a$ is the prime power $p^{\alpha}$.
Then $L(a)$ increases by a factor of $p$,
that is,
$L(a) = p \cdot L(a-1)$.
If $M(a-1) = M(a)$ 
then $L(a)/M(a) = p \cdot L(a-1)/M(a-1) \geq L(a-1)/M(a-1)$. 
On the other hand, if $M(a) \neq M(a-1)$ 
then we have that $M(a) = p^{\alpha-1}$.
Note now that $M(a-1) \geq p^{\alpha-2}$.
We thus have
\begin{align*}
     \frac{L(a-1)}{M(a-1)}
& =
     \frac{p \cdot L(a-1)}{p \cdot M(a-1)}
     \leq
     \frac{L(a)}{p \cdot p^{\alpha-2}}
=
     \frac{L(a)}{M(a)} . 
\qedhere
\end{align*}
\end{proof}

Using Lemma~\ref{lemma_L_M},  we now remove
the lower bound $a \cdot k$ on $n$.
\begin{theorem}
For an integer $a> 1$, let $L(a)$ and $M(a)$ be as in
Proposition~\ref{proposition_M}.
Let $k$ be an integer greater than or equal to
the quotient $L(a)/M(a)$,
and let $n$ be a positive integer 
bounded by  $n \leq (a+1) \cdot k$.
Then the van der Waerden complex $\vdW(n,k)$
is contractible.
\label{theorem_main_bound}
\end{theorem}
\begin{proof}
Let $b \leq a$ be a positive integer
such that
$b \cdot k < n \leq (b+1) \cdot k$.
By Lemma~\ref{lemma_L_M}
we have that
$L(b)/M(b) \leq L(a)/M(a) \leq k$.
Hence applying 
Proposition~\ref{proposition_M}
with the parameter~$b$ implies that
$\vdW(n,k)$ is contractible.
\end{proof}

\begin{remark}
{\rm
If $a+1$ is not a prime power, then
the statement of
Theorem~\ref{theorem_main_bound}
can be sharpened by increasing $a$ by $1$, because we have that
$L(a) = \lcm(1,2,\ldots,a) = \lcm(1,2,\ldots,a+1)$.
Hence in this case we have that  
the complex $\vdW(n,k)$ is contractible for
$k \geq L(a)/M(a)$ and $n \leq (a+2)k$.
}
\end{remark}

We have the following lower and upper 
bounds on $M(a)$.
\begin{lemma}
For an integer $a >1$, let $M(a)$ be defined as in
Proposition~\ref{proposition_M}. Then we have that
$$
     {a}/{4} < M(a) \leq {a}/{2}.
$$
\end{lemma}
\begin{proof}
Assume that
$M(a) = p^{\alpha - 1}$
where $p$ is a prime and $\alpha > 0$.
We have that
$p^{\alpha} \leq a$
and hence
$M(a) = p^{\alpha-1} \leq  {p^{\alpha}}/{2} \leq {a}/{2}$.
As for the lower bound,
let $\beta \geq 0$ satisfy
 $2^{\beta} \leq a < 2^{\beta+1}$.
Then we have
${a}/{4} < 2^{\beta - 1} \leq M(a)$. 
\end{proof}

\section{Examples}
\label{section_examples}

For the sake of illustration, 
in this section
we discuss
the topology of the family of van der Waerden complexes 
$\vdW(5k,k)$ for $k \geq 1$.

Let $L(a)$ and $M(a)$ be defined as in Proposition~\ref{proposition_M}. 
Then $L(4) = 12 = 2^2 \cdot 3 $ and accordingly $M(4) = 2$. 
Applying Theorem~\ref{theorem_main_bound} for $a=4$ 
we conclude the complex $\vdW(5k,k)$ 
is contractible for all $k \geq L(4)/M(4) = 6$.
In fact, when $k \geq 6$
the van der Waerden complex  $\vdW(n,k)$ 
is contractible for all positive integers $n \leq 5k$.

The remaining van der Waerden complexes in the family,
namely 
$\vdW(5,1)$
through
$\vdW(25,5)$,  are not contractible. 
For instance,
the complex $\vdW(5,1)$ is the complete graph $K_5$ 
which is homotopy equivalent to a wedge of $\binom{4}{2} = 6$ circles. 
To determine the homotopy equivalence of
the other four van der Waerden complexes, we construct 
on each complex an acyclic matching with a minimal number of critical cells.

\begin{example}
{\rm
The complex $\vdW(10,2)$ has a discrete
Morse matching with the
critical cells~$\{10\}$ and
$\{x,x+3\}$
for $1 \leq x \leq 7$.
Thus this complex is
homotopy equivalent to
a wedge of seven circles.
}
\label{example_10_2}
\end{example}
\begin{proof}
On the face poset of 
$\vdW(10,2) - \{\emptyset\}$,
match the singleton $\{x\}$
with $\{x,x+1\}$
for $1 \leq x \leq 9$.
Furthermore, match
$\{x,x+2d\}$ with
$\{x,x+d,x+2d\}$
for $1 \leq x < x+2d \leq 10$.
It is straightforward
to see that this matching is
acyclic and the critical
cells have the form $\{x,x+3\}$
for $1 \leq x \leq 7$.
The homotopy result follows from
Proposition~\ref{proposition_wedge_of_spheres}.
\end{proof}

Let $Q$ 
and
$\varphi : \vdW(n,k) - \{\emptyset\} \longrightarrow Q$
be the poset and the order-preserving map
which appear in the beginning of the proof of
Proposition~\ref{proposition_L}.
We will use them in
Examples~\ref{example_15_3}
through~\ref{example_25_5}.

\begin{example}
{\rm
The complex $\vdW(15,3)$ has a discrete
Morse matching with the
critical cells
$\{15\}$ and
$\{x, x+3, x+6\}$
for $1 \leq x \leq 9$.
Hence this complex is
homotopy equivalent to
a wedge of nine
$2$-dimensional spheres.
}
\label{example_15_3}
\end{example}
\begin{proof}
For a non-empty face
$F \in \vdW(15,3)$,
the difference 
$\max(F) - \min(F)$ takes
one of the following values:
$0$, $1$, $2$, $3$, $4$, $6$, 
$8$, $9$ and $12$.
Consider the following
matching.
\begin{enumerate}
\item
If the difference
$\max(F)-\min(F) \in \{0, 1, 2, 4, 8\}$,
we match the face $F$ in the same manner as
in Example~\ref{example_10_2}.
This leaves $\{15\}$
as a critical cell.

\item
If $1 \leq x < x+3d \leq n$, we match the face $F$ in the
fiber $\varphi^{-1}((x,x+3d))$
with the symmetric difference
$F \triangle \{x+d\}$.
Note that when
the difference is $6$,
this leaves the critical cells
$\{x,x+3,x+6\}$
for
$1 \leq x \leq 9 = n-6$.
\end{enumerate}
Observe that all of these matchings
stay inside the corresponding
fiber of
the poset map $\varphi$
and are acyclic in each fiber.
Hence
by Theorem~\ref{theorem_patchwork}
their union is an acyclic matching.
The critical cells are~$\{15\}$ and
$\{x, x+3, x+6\}$
for $1 \leq x \leq 9$.
It follows from Proposition~\ref{proposition_wedge_of_spheres}
that $\vdW(15,3)$ is homotopy equivalent to a wedge of nine spheres.
\end{proof}

\begin{table}[t]
$$
\begin{array}{ c | c c c}
\vdW(5k,k) & \text{homotopy type}
\\ \hline
\vphantom{\rule{0 mm}{5 mm}}
\vdW(5,1) 
\vphantom{\rule{0 mm}{5 mm}}
& (S^{1})^{\vee 6}
\\
\vdW(10,2) & (S^{1})^{\vee 7}
\\
\vdW(15,3) & (S^{2})^{\vee 9}
\\
\vdW(20,4) & (S^{2})^{\vee 22} 
\\
\vdW(25,5) & (S^{2})^{\vee 32}
\end{array}
$$
\caption{The homotopy types
for~$\vdW(5k,k)$
such that $1 \leq k \leq 5$.
By Theorem~\ref{theorem_main_bound}
the complex $\vdW(5k,k)$
is contractible for $k \geq 6$.}
\label{table_5k_k}
\end{table}

\begin{example}
{\rm
The complex
$\vdW(20,4)$ has a discrete
Morse matching with the
critical cells~$\{20\}$,
$\{x, x+3, x+6\}$
for $1 \leq x \leq 14$
and
$\{x,x+4,x+12\}$
for $1 \leq x \leq 8$.
Hence this complex is
homotopy equivalent to
a wedge of $22$
$2$-dimensional spheres.
}
\label{example_20_4}
\end{example}
\begin{proof}
Now we have that
for a non-empty face $F \in \vdW(20,4)$, the difference
$\max(F)-\min(F) \in \{0, 1, 2, 3, 4, 6, 8, 9, 12, 16\}$.
\begin{enumerate}
\item
When 
$\max(F)-\min(F) \in \{0, 1, 2, 3, 6, 9\}$,
we match in the same manner as 
in Example~\ref{example_15_3}.
above. This leaves 
one zero-dimensional
critical cell $\{20\}$ and
$14$ two-dimensional critical cells 
of the form
$\{x, x+3, x+6\}$
for $1 \leq x \leq 14$.

\item
For $1 \leq x \leq 16$, 
match each face $F$
in the fiber
$\varphi^{-1}((x, x+4))$
with the symmetric difference
$F \triangle \{x + 2\}$.

\item
For $1 \leq x \leq 12$,
match each face $F$
in the fiber
$\varphi^{-1}((x, x+8))$
with the symmetric difference
$F \triangle \{x + 4\}$.

\item
For $1 \leq x \leq 8$
and
$F \in \varphi^{-1}((x,x+12))$,
there are three cases to
consider.
If $F \triangle \{x+6\}
\in \vdW(20,4)$,
then we match $F$ with
the symmetric difference
$F \triangle \{x+6\}$.
The remaining cases are when
$F$ is one of the following
three faces:
$\{x, x+4, x+12\}$,
$\{x, x+8, x+12\}$ or
$\{x, x+4, x+8, x+12\}$. We proceed by
matching the last two faces
and leaving
$\{x, x+4, x+12\}$
as a critical face, for $1 \leq x \leq 8$. 

\item
For $1 \leq x \leq 4$, we
match each face $F \in \varphi^{-1}((x, x+16))$
with the symmetric difference
$F \triangle \{x + 8\}$.
\end{enumerate}
The result follows by the same
reasoning as the end of
Example~\ref{example_15_3}.
\end{proof}

\begin{example}
{\rm
The complex
$\vdW(25,5)$ has a discrete
Morse matching with the
critical cells~$\{25\}$,
$\{x, x+3, x+6\}$
for $1 \leq x \leq 19$
and
$\{x,x+4,x+12\}$
for $1 \leq x \leq 13$.
Hence this complex is
homotopy equivalent to
a wedge of $32$
$2$-dimensional spheres.
}
\label{example_25_5}
\end{example}
\begin{proof}
Note that for a non-empty face $F$
we have that the difference
$\max(F) - \min(F)$ belongs to the set
$\{0, 1, 2, 3, 4, 5, 6, 8, 9, 10, 11, 12, 15, 16, 20\}$.
Consider the following matching:
\begin{enumerate}
\item
When the difference
$\max(F)-\min(F) \in \{0,1,2,3,4,6,8,9,12,16\}$,
we match
as in Example~\ref{example_20_4}.
This leaves the
critical cells
$\{25\}$,
$\{x, x+3, x+6\}$
for $1 \leq x \leq 19$
and
$\{x,x+4,x+12\}$
for $1 \leq x \leq 13$.

\item
For $1 \leq x < x+5d \leq 25$
(which implies
$1 \leq d \leq 4$), 
match the face $F \in \varphi^{-1}((x,x+5d))$
with the symmetric difference
$F \triangle \{x+d\}$.
\end{enumerate}
The same reasoning as
in the previous examples 
yields the result.
\end{proof}

Note that since $3$ and $5$ are
prime, it was easier to describe
the matchings for
the complexes
$\vdW(15,3)$
and
$\vdW(25,5)$.

\section{Concluding remarks}

Some natural questions arise concerning
the topology of  the van der Waerden complex.
If $\vdW(n,k)$ is noncontractible, is it
always homotopy equivalent to a wedge of spheres?
There is some evidence for this in the examples
described in Section~\ref{section_examples}.
One possible way to prove such a result would be
to construct an acyclic matching that satisfies 
Kozlov's alternating-path condition~\cite{Kozlov}.

When the complex $\vdW(n,k)$
is noncontractible, can one predict its Betti numbers 
or say something nontrivial about their behavior 
with respect to $n$ or $k$? Moreover, do sequences of these 
Betti numbers have an 
underlying connection to other mathematical structures, including
number-theoretic ones?

What is the error term in the asymptotic bound 
$\log k / \log \log k$ for the dimension of the homotopy
type in Theorem~\ref{asymptotic}? 
Given that the Prime Number Theorem is fundamentally 
encoded into this asymptotic via the primorial function, 
it is conceivable that a successful analysis 
in this direction may yield a statement about the 
homotopy dimension bound of $\vdW(n,k)$ 
that is equivalent to the Riemann Hypothesis.

Are Theorems~\ref{theorem_main}
and~\ref{theorem_main_bound} provably
the tightest possible general bounds for the dimension (with respect
to $k$) and for contractibility (with respect to $n/k$).
Can one generalize Theorem~\ref{theorem_main_bound}
to a bound on the
dimension of the homotopy
type of $\vdW(n,k)$ with respect to both $n$ and~$k$, as opposed
to just~$k$, as in Theorem~\ref{theorem_main}?

\section*{Acknowledgments}

The authors thank Nigel Pitt for discussions related
to asymptotics in Section~\ref{section_upper_bound}.
The authors also thank the referee for providing 
the references~\cite{Bjorner,Musiker_Reiner,Pakianathan_Winfree}.
The first author was partially supported by
National Security Agency grant~H98230-13-1-0280.
This work was partially supported by a grant from
the Simons Foundation (\#206001 to Margaret~Readdy).
The first and fourth authors thank 
the Princeton University Mathematics Department
where this work was initiated.

\newcommand{\journal}[6]{{\sc #1,} #2, {\it #3} {\bf #4} (#5), #6.}
\newcommand{\book}[5]{{\sc #1,} #2, #3, #4, #5.}
\newcommand{\preprint}[5]{{\sc #1,} #2, #3, #4, #5.}
\newcommand{\arxiv}[3]{{\sc #1,} #2, #3.}
\newcommand{\thesis}[4]{{\sc #1,} ``#2,'' Doctoral dissertation, #3,~#4.}

\bigskip

\noindent
{\em R.\ Ehrenborg, M.\  Readdy,
Department of Mathematics,
University of Kentucky,
Lexington, KY 40506,}
{\tt richard.ehrenborg@uky.edu},
{\tt margaret.readdy@uky.edu},

\noindent
{\em L.\ Govindaiah, P.\ S.\ Park,
Department of Mathematics,
Princeton University,
Princeton, NJ 08540,}
{\tt likithg@princeton.edu},
{\tt pspark@math.princeton.edu}.

\end{document}